\documentclass[12pt]{amsart}
\usepackage{graphicx}
\newtheorem{thm}{Theorem}[section]
\newtheorem{lem}[thm]{Lemma}
\newtheorem{cor}[thm]{Corollary}
\newtheorem{prop}[thm]{Proposition}

\theoremstyle{definition}

\theoremstyle{remark}

\begin{document}

\title[Constructing ordered orthogonal arrays via sudoku]{Constructing ordered orthogonal arrays via sudoku}
\author{John Lorch}
\address{Department of Mathematical Sciences\\ Ball State University\\Muncie, IN  47306-0490}
\email{jlorch@bsu.edu}
\date{December 19, 2014}
\begin{abstract}
For prime powers $q$ we use ``strongly orthogonal" linear sudoku solutions of order $q^2$ to construct ordered orthogonal arrays of type ${\rm OOA} (4,s,2,q)$, and for each $q$ we present a range of values of $s$ for which these constructions are valid. These results rely strongly on flags of subspaces in a four dimensional vector space over a finite field.
\end{abstract}
\maketitle

\def\a{{\bf a}}
\def\MF{M^{2\times 2}(\F )}
\def\rtimes{{\times\!\! |}}
\def\rk{{\rm rank}}
\def\A{{\mathcal A}}
\def\reg{{\mathcal R}}
\def\P{{\mathcal P}}
\def\F{{\mathbb F}}
\def\L{{\mathcal L}}
\def\R{{\mathbb R}}
\def\C{{\mathbb C}}
\def\Z{{\mathbb Z}}
\def\B{{\mathcal B}}
\def\G{{\mathcal G}}
\def\g{\mathfrak g}
\def\OOA{{\rm OOA}}
\def\OA{{\rm OA}}
\def\S{{\mathcal S}}
\def\SA{{\rm SA}}
\def\SD{{\rm SD}}
\def\TD{{\rm TD}}
\def\Q{{\mathcal Q}}
\def\pt#1{{\langle  #1 \rangle}}
\def\ds{\displaystyle}
\def\v{{\bf v}}
\def\w{{\bf w}}

\section{Introduction}
\label{s:introduction}

\subsection{Purpose}
It has long been known that families of mutually orthogonal latin squares have important connections and applications to finite geometry, statistical design, and graph theory, many of which are illustrated in \cite{cC07} and \cite{fR84}. This is exemplified by the following Theorem (e.g., see \cite{cC07} or \cite{cC01}).
\begin{thm}\label{t:connections}
Let $n,s\in \Z ^+$ with $3\leq s \leq n+1$. The following are equivalent:
\begin{itemize}
\item[(a)] There exists a collection of $s-2$ mutually orthogonal latin squares of order $n$.
\item[(b)] There exists a Bruck net of order $n$ and degree $s$.
\item[(c)] There exists a transversal design of type $\TD (s,n)$.
\item[(d)] There exists an orthogonal array of type $\OA (s,n)$.
\item[(e)] There exists an edge partition of the complete $s$-partite graph $K_{n,\dots ,n}$ into complete subgraphs of order $s$.
\end{itemize}
\end{thm}
A sudoku solution is a special kind of latin square. Given that families of mutually orthogonal sudoku solutions have become of interest in their own right (e.g., see \cite{rB08}, \cite{aK07}, \cite{aK10}, \cite{jL10}, \cite{jL12}, \cite{rP09}), it is reasonable to seek combinatorial `soul mates' for orthogonal sudoku solutions, along the lines of those given for latin squares in Theorem \ref{t:connections}. Ordered orthogonal arrays of
type $\OOA (4,s,2,q)$ seem to make good candidates: In this paper we show that when $q$ is a prime power, special families of mutually orthogonal sudoku solutions of order $q^2$ can be used to construct ordered orthogonal arrays of this type.

\subsection{Background}
For a positive integer $n$, a {\bf sudoku solution} of order $n^2$ is an $n^2\times n^2$ Latin square with the additional requirement that every symbol must appear in each canonical $n\times n$ subsquare. These occur in case $n=3$ when one is lucky enough to solve the newspaper sudoku puzzle (Figure \ref{f:sudokusolution}). Locations in a sudoku solution of order
$n^2$ can be populated by any set of $n^2$ distinct symbols; we use symbols $\{0,1,\dots ,n^2-1\}$.

\begin{figure}[h]
      $$ {\small \begin{array}{| c  c  c | c c c | c c c|}
    \hline
      0 & 1 & 2 &  4 & 5 &3 & 8 & 6 & 7 \\
      3 & 4 & 5 &  7 & 8 &6 & 2 & 0 & 1 \\
      6 & 7 & 8 &  1 & 2 &0 & 5 & 3 & 4 \\ \hline
      1 & 2 & 0 &  5 & 3 &4 & 6 & 7 & 8 \\
      4 & 5 & 3 &  8 & 6 &7 & 0 & 1 & 2 \\
      7 & 8 & 6 &  2 & 0 &1 & 3 & 4 & 5 \\ \hline
      2 & 0 & 1 &  3 & 4 &5 & 7 & 8 & 6 \\
      5 & 3 & 4 &  6 & 7 &8 & 1 & 2 & 0 \\
      8 & 6 & 7 &  0 & 1 &2 & 4 & 5 & 3 \\ \hline
    \end{array}} $$
       \caption{A sudoku solution of order $9$. }
       \label{f:sudokusolution}
       \end{figure}
Two latin squares of like order are said to be {\bf orthogonal} if, upon superimposition, each ordered pair of symbols appears exactly once (Figure \ref{f:orthogonal}). Families of pairwise orthogonal latin squares have long been an object of study, in part due to connections illustrated in Theorem \ref{t:connections}. One of the most famous open problems concerning orthogonal latin squares is determining the maximum size $N(n)$ of a family of pairwise mutually orthogonal latin squares of order $n$. It is well known that $N(n)\leq n-1$, and, due to field theoretic considerations, that $N(n)=n-1$ when $n$ is a prime power (e.g., \cite{fR84} and \cite{gM95}). Orthogonality questions for latin squares specialize to sudoku. For instance, it is known that the maximum size of a family of pairwise mutually orthogonal sudoku solutions of order $n^2$ is $n(n-1)$, and that this bound is achieved when $n$ is a prime power (e.g., \cite{rB08}, \cite{jL12}, \cite{rP09}). {\bf Linear sudoku solutions}, which are useful for creating mutually orthogonal families of sudoku solutions, play a key role in the present work and will be defined carefully in Section \ref{s:linear}.

\begin{figure}[h]
$$
\begin{minipage}{1in}
    $\begin{array}{| c c | c c |}
    \hline
      0 & 1 & 3 &2 \\
      2 & 3 & 1 & 0\\ \hline
      3 &2 & 0 & 1 \\
      1 & 0 & 2 &3 \\ \hline
    \end{array}$
  \end{minipage}
  \begin{minipage}{1in}
    $\begin{array}{|c c | c c|}
    \hline
      0 & 3 & 2 &1 \\
      2 & 1 & 0 & 3\\ \hline
      3 &0 & 1 & 2 \\
      1  & 2 &3 & 0\\ \hline
    \end{array}$
  \end{minipage}
  \quad \Longrightarrow \quad
  \begin{array}{| c c | c c |}
    \hline
      00 & 13 & 32 &21 \\
      22 & 31 & 10 & 03\\ \hline
      33 &20 & 01 & 12 \\
      11 & 02 & 23 &30 \\ \hline
    \end{array}
   $$
\caption{Orthogonal sudoku solutions of order four.}
\label{f:orthogonal}
\end{figure}

Part (d) of Theorem \ref{t:connections} indicates that families of mutually orthogonal latin squares are characterized by orthogonal arrays. In the spirit of this theorem, we explore connections between mutually orthogonal sudoku solutions and ordered orthogonal arrays.  An {\bf ordered orthogonal array} of type $\OOA _\lambda (t,s,\ell ,v)$ consists of an $s\ell \times \lambda v^t$ array whose entries are drawn from an alphabet of size $v$. This array must satisfy a condition that requires explanation: Rows of the array are partitioned into $s$ bands, each of width $\ell$, and can be labeled lexicographically from top to bottom by ordered pairs $\{(i,j)\mid 1\leq i\leq s, 1\leq j\leq \ell\}$. A set $T$ of $t$ rows within this array is said to be {\bf top-justified} if whenever the $(i,j)$-th row lies in $T$ we have that the $(i,k)$-th row also lies in $T$ for $1\leq k\leq j$. Regarding a set $T$ of top justified rows as a sub-array of the original array, we require that each $t$-tuple column of $T$ occurs exactly $\lambda$ times in this subarray. (See Figure \ref{f:ooa} for an example.) Ordered orthogonal arrays are generalizations of orthogonal arrays: one obtains an `ordinary' orthogonal array by putting $\ell =\lambda =1$ and $t=2$. We will focus on the case $\ell =2$, $\lambda =1$, and $t=4$. We follow the custom of dropping the $\lambda$ subscript from $\OOA$ when $\lambda =1$.

\begin{figure}[h]
$$
    \begin{array}{|cccccccccccccccc|}
 \hline
 {\bf 0}& \bf  0 & \bf 0 & \bf 0 &\bf  0 & \bf 0 &\bf  0 & \bf 0 & \bf 1 & \bf 1 &\bf  1 & \bf 1 & \bf 1 &\bf  1 &\bf  1 &\bf  1 \\
  \bf 0&\bf  0 &\bf  0 &\bf 0 & \bf 1 & \bf 1 & \bf 1 & \bf 1 & \bf 0 & \bf 0 & \bf 0 &\bf  0 & \bf 1 & \bf 1 &\bf  1 &\bf  1 \\
 \hline
 \bf 0& \bf 0 & \bf 1 &\bf 1 &\bf 0 &\bf 0 &\bf 1 &\bf 1 &\bf 0 &\bf 0 &\bf 1 &\bf 1 &\bf 0 &\bf 0 &\bf 1 &\bf 1 \\
 0& 1 & 0 & 1 & 0 & 1 & 0 & 1 & 0 & 1 & 0 & 1 & 0 & 1 & 0 & 1 \\
 \hline
\bf 0&\bf 1 &\bf 1 &\bf 0 &\bf 1 &\bf 0 &\bf 0 &\bf 1 &\bf 0 &\bf 1 &\bf 1 &\bf 0 &\bf 1 &\bf 0 &\bf 0 &\bf 1 \\
 0& 0 & 1 & 1 & 1 & 1 & 0 & 0 & 1 & 1 & 0 & 0 & 0 & 0 & 1 & 1 \\
 \hline
 \end{array}
   $$
\caption{An ordered orthogonal array of type $\OOA  (4,3,2,2)$. A set of top-justified rows is shown in bold.}
\label{f:ooa}
\end{figure}

 Ordered orthogonal arrays were first defined for arbitrary parameters in \cite{kL96}, where for certain parameters they are used to characterize $(t,m,s)$-nets, which in turn have significant applications in quasi-Monte Carlo methods. Constructions and/or parameter bounds for ordered orthogonal arrays and associated $(t,m,s)$-nets  may be found, for example, in \cite{kL96}, \cite{mS99}, \cite{mA99}, and \cite{hT10}.\footnote{Several papers, including Lawrence's original paper \cite{kL96}, refer to ordered orthogonal arrays as ``generalized orthogonal arrays." We follow the terminology of \cite{cC07}. Also, $(t,m,s)$ nets are ``low discrepancy" multisets in $[0,1)^s$ and should not be confused with Bruck nets. The ordered orthogonal arrays considered in this paper do not characterize $(t,m,s)$ nets.}

When $q$ is a prime power, we construct ordered orthogonal arrays of type $\OOA (4,s,2,q)$ for $s\leq \frac{q+4}{2}$. (Theorem \ref{t:big} and Corollary \ref{c:big}). These are not the best possible existence results for small values of $q$, such as $q=3,4,5$, which necessarily result in $s<6$: In Theorem \ref{t:substrong} and Corollary \ref{c:little} we construct arrays of type  $\OOA (4,3,2,q)$  and $\OOA (4,4,2,q)$. Our constructions arise by identifying these ordered orthogonal arrays with special sets of mutually orthogonal sudoku solutions that we call {\bf strongly orthogonal}.\footnote{The term {\em strong orthogonality} has been used elsewhere in the world of combinatorial design, such as in \cite{eM13}, where the spirit
 of its use, and perhaps more, is reminiscent of how the term is used here.} Results in this paper rely strongly on flags of subspaces in a four dimensional vector space over a finite field.
 
 We briefly attempt to put these results in context. Let $f(k,b)$ denote the maximum value of $s$ for which $\OOA (4,s,k,b)$ exists. In general, the value of $f(k,b)$ is unknown, and a principle goal of combinatorial design research, in addition to constructing designs, is to produce ever larger lower bounds for unknown maximal parameters such as $f(k,b)$. Bush \cite{kB52} showed that $\OOA (t,q+1,1,q)$ exists when $q$ is a prime power with $q>t$, so in particular $f(1,q)\geq q+1$ for $q>4$. Meanwhile, Lawrence \cite{kL96} showed that $f(4,q)\geq \lfloor \frac{f(1,q)}{2} \rfloor$. Therefore, since $f(4,q)\leq f(2,q)$, we can conclude that $f(2,q)\geq \lfloor \frac{q+1}{2}\rfloor $ for prime powers $q>4$. Observe that the results mentioned in the previous paragraph indicate that $f(2,q)\geq \frac{q+4}{2}$ for prime powers $q$. So perhaps we have found a larger lower bound for $f(2,q)$, but the real novelty here is the method of construction of ordered orthogonal arrays, which might find applications in other settings.

The structure of the paper is as follows: In Section \ref{s:sudokuarray} we identify mutually orthogonal families of sudoku solutions with arrays that have the potential to be ordered orthogonal arrays (Theorem \ref{t:equivtosudoku}). In Section \ref{s:so} we tease out necessary and sufficient conditions that must be satisfied by a collection of orthogonal sudoku solutions in order to give rise to an ordered orthogonal array (Proposition \ref{p:conditions}). The conditions of Section \ref{s:so} are translated to a linear setting in Section \ref{s:linear}, and sudoku flags are introduced. Using the machinery of sudoku flags, construction of ordered orthogonal arrays occurs in Section \ref{s:co}.

\section{Sudoku arrays}\label{s:sudokuarray}

In this section we produce an exact analog of Theorem \ref{t:connections}, part (d), for sudoku. We will also observe that every ordered orthogonal array of type $\OOA  (4,s,2,n)$ determines a special family of $s-2$ mutually orthogonal sudoku solutions of order $n^2$.

We briefly introduce terminology that will be used frequently as we proceed. If $b_nb_1$ is the base-$n$ representation of $x\in \{0,1,\dots, n^2-1\}$, meaning $x=b_n\cdot n+b_1\cdot 1$ and $b_n,b_1\in\{0,1,\dots, n-1\}$, we call $b_n$ the {\bf radix digit} of $x$ and $b_1$ the {\bf units digit}.
Locations within an order-$n^2$ sudoku solution can be identified with four-tuples $(x_1,x_2,x_3,x_4)\in \Z_n ^4$ where
$x_1$ determines a {\bf large row} (i.e., a row of subsquares),
$x_2$ determines a row within a large row,
$x_3$ determines a {\bf large column} (i.e., a column of subsquares), and
$x_4$ determines a column within a large column.
Large rows (and rows within large rows) are labeled in increasing order from top to bottom starting with $0$, while large columns (and columns within large columns) are labeled in increasing order from left to right starting with $0$. Entries $x_1$ and $x_2$ together completely determine the row of a particular location; entries $x_3$ and $x_4$ do the same for the column of a particular location. Said another way, the first two coordinates of $(x_1,x_2,x_3,x_4)$ can be viewed as a base-$n$ representation of a row location, and the latter two coordinates as a base-$n$ representation of a column location. For example, in Figure
\ref{f:sudokusolution} the ``$4$" lying in the third row and rightmost column has location $(1,0,2,2)$.

 Families of mutually orthogonal sudoku solutions can be characterized by {\bf sudoku arrays}, which bear great similarity to ordered orthogonal arrays of type $\OOA  (4,s,2,n)$.
 For $n,s\in \Z ^+$ with $s\geq 3$, a sudoku array of type $\SA (s,n)$ consists of an $2s \times  n^4$ array whose entries are drawn from an alphabet of size $n$. Rows of the array are partitioned into $s$ bands, each of width $2$, and can be labeled lexicographically, just as for ordered orthogonal arrays.  A set $T$ of $4$ rows within this array is said to be {\bf sudoku top-justified} if it is top justified and, for $i\geq 3$, if the $(i,1)$-th row lies in $T$ then the $(i,2)$-th row also lies in $T$.\footnote{For sudoku top-justified sets $T$: In a given band below the top two row bands, either {\em both} rows lie in $T$, or neither of the two rows lie in $T$.} Regarding a set $T$ of  sudoku top-justified rows as a sub-array of the original array, we require that each $4$-tuple column of $T$ occurs exactly once in this subarray. Rows in the top two bands of the array are called {\bf location rows}. The bottom array in Figure \ref{f:sudaltarray} is an $\SA (4,2)$.

 Consider the two sudoku solutions at the top of Figure \ref{f:sudaltarray}. They are orthogonal, and together can be used to construct the sudoku array shown in the figure. The bold $2$ and bold $3$ both lie in location $(0,0,1,1)$. Writing $2$ and $3$ in base $n=2$ gives $(1,0)$ and $(1,1)$, respectively. If we concatenate the location and symbol data we obtain $(0,0,1,1,1,0,1,1)$, and this is precisely the bold column in the sudoku array. All other columns in the sudoku array are obtained similarly: the location rows correspond to locations in a sudoku solution, while the other rows correspond to symbols expressed in base $n$. In general we have the following analog of Theorem \ref{t:connections}, part (d):

 \begin{figure}[h]
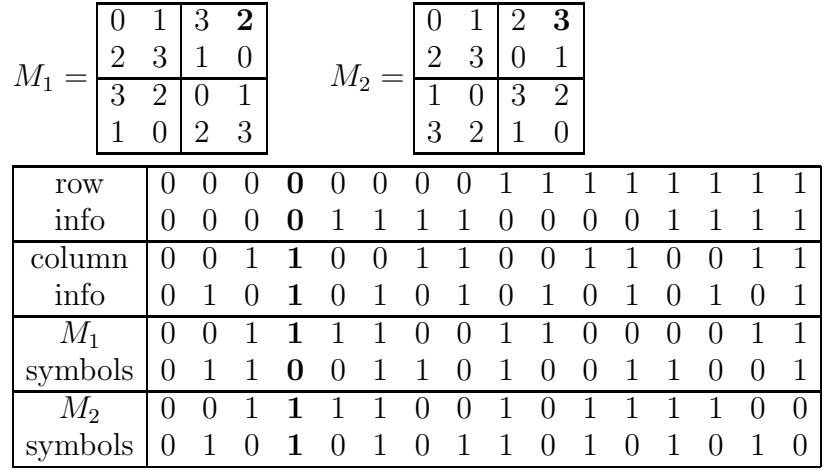

 \begin{align*}
    &M_1 =\begin{array}{|cc|cc|}
    \hline
 0 & 1 & 3 & {\bf 2} \\
 2 & 3 & 1 & 0 \\
 \hline
 3 & 2 & 0 & 1 \\
 1 & 0 & 2 & 3 \\
 \hline
    \end{array}
    \qquad
   M_2= \begin{array}{|cc|cc|}
    \hline
 0 & 1 & 2 & {\bf 3} \\
 2 & 3 & 0 & 1 \\
 \hline
 1 & 0 & 3 & 2 \\
 3 & 2 & 1 & 0 \\
 \hline
    \end{array} \\
    &\begin{array}{|c |cccccccccccccccc|}
 \hline
{\rm row}& 0& 0 & 0 & {\bf 0 }& 0 & 0 & 0 & 0 & 1 & 1 & 1 & 1 & 1 & 1 & 1 & 1 \\
{\rm info}& 0& 0 & 0 & {\bf 0 } & 1 & 1 & 1 & 1 & 0 & 0 & 0 & 0 & 1 & 1 & 1 & 1 \\
 \hline
{\rm column}& 0& 0 & 1 & {\bf 1} & 0 & 0 & 1 & 1 & 0 & 0 & 1 & 1 & 0 & 0 & 1 & 1 \\
{\rm info}& 0& 1 & 0 & {\bf 1 } & 0 & 1 & 0 & 1 & 0 & 1 & 0 & 1 & 0 & 1 & 0 & 1 \\
 \hline
         M_1 & 0& 0 & 1 & {\bf 1} & 1 & 1 & 0 & 0 & 1 & 1 & 0 & 0 & 0 & 0 & 1 & 1 \\
{\rm symbols}& 0& 1 & 1 & {\bf 0 }& 0 & 1 & 1 & 0 & 1 & 0 & 0 & 1 & 1 & 0 & 0 & 1 \\
 \hline
         M_2 &  0& 0 & 1 & {\bf 1} & 1 & 1 & 0 & 0 & 1 & 0 & 1 & 1 & 1 & 1 & 0 & 0 \\
{\rm symbols}&  0& 1 & 0 & {\bf 1} & 0 & 1 & 0 & 1 & 1 & 0 & 1 & 0 & 1 & 0 & 1 & 0 \\
 \hline
 \end{array}
    \end{align*}
       \caption{Two orthogonal sudoku solutions and a corresponding sudoku array.}
       \label{f:sudaltarray}
       \end{figure}

\begin{thm}\label{t:equivtosudoku}
Let $s\geq 3$. The existence of an $\SA (s,n)$ is equivalent to the existence of $s-2$ mutually orthogonal sudoku solutions of order $n^2$.
\end{thm}

\begin{proof}
We describe how a set of $s-2$ mutually orthogonal sudoku solutions $\{M_1,\dots, M_{s-2}\}$ of order $n^2$ implies a sudoku array $A$ of type $\SA (s,n)$; it will be apparent that this process is reversible.

Let the entry in location $((i,k),m)$ of $A$ be denoted $a_{ikm}$, where $1\leq i\leq s$, $1\leq k\leq 2$, $1\leq m\leq n^4$, and $a_{ikm}\in \Z_n$.
According to earlier discussion, the $n^4$ locations of each sudoku solution can be described by $n^4$ distinct $4$-tuples of the form $(a_{11m},a_{12m},a_{21m},a_{22m})\in \Z_ n^4$. Together these $4$-tuples form the first four rows (two bands) of $A$. From top to bottom these location rows will be denoted $l_{r,1},l_{r,2},l_{c,1},l_{c,2}$.

Relabel the sudoku solutions so that the symbols are base-$n$ representations of the integers $\{0,1,\dots ,n^2-1\}$. (Note that relabeling does not affect orthogonality.) Let
$$
(\mu^i_1(a_{11m},a_{12m},a_{21m},a_{22m}),\mu ^i _2 (a_{11m},a_{12m},a_{21m},a_{22m}))
   $$
denote the symbol of $M_i$ lying in position
$(a_{11m},a_{12m},a_{21m},a_{22m})$ where $\mu ^i_1$ is the radix digit and $\mu ^i_2$ is the units digit. Then put $a_{i+2,j,m}= \mu ^{i} _j (a_{11m},a_{12m},a_{21m},a_{22m})$ for $1\leq i \leq s-2$. This completes the description of $A$.

To show that $A$ is a sudoku array we need to verify that every set $T$ of sudoku top-justified rows satisfies the non-repetition condition. There are several cases to consider, where $i,j$ are distinct members of $\{1,\dots ,s-2\}$ and $R_{i,k}$ denotes the $(i+2,k)$-th row of $A$:
\begin{itemize}
\item $T=\{R_{i,1},R_{i,2},R_{j,1},R_{j,2}\}$: Any repetition here indicates that the same ordered pair of symbols (one symbol from $M_{i}$ and another from $M_{j}$) appear in at least two distinct locations, contradicting the orthogonality of $M_{i}$ and $M_{j}$. (This case does not occur when $s=3$.)
\item $T=\{l_{r,1},l_{r,2},R_{i,1},R_{i,2}\}$: Any repetition here implies that $M_{i}$ has a repeated symbol in some row, a contradiction.
\item $T=\{l_{c,1},l_{c,2},R_{i,1},R_{i,2}\}$: Any repetition here implies  that $M_{i}$ has a repeated symbol in some column, a contradiction.
\item $T=\{l_{r,1},l_{r,2},l_{c,1},l_{c,2}\}$: Any repetition here indicates a repeated location, something that is ruled out in our hypothesis.
\item $T=\{l_{r,1},l_{c,1},R_{i,1},R_{i,2}\}$: Any repetition here implies repetition of a symbol within some subsquare of $M_{i}$, a contradiction.
\end{itemize}
We conclude that $A$ is a sudoku array.
\end{proof}

Observe that every ordered orthogonal array of type $\OOA _1 (4,s,2,n)$ is a sudoku array of type $\SA (s,n)$, and hence every ordered orthogonal array of this type determines a collection of $s-2$ mutually orthogonal sudoku solutions of order $n^2$ (Theorem \ref{t:equivtosudoku}). However, the reverse implication is false: for example, the array in Figure \ref{f:sudaltarray} fails to be an $\OOA (4,4,2,2)$, as can be seen by considering the set of four rows obtained by selecting the top row from each of the four bands.

\section{Strongly orthogonal sudoku}\label{s:so}
Here we describe additional conditions to be imposed on a family of $s-2$ mutually orthogonal sudoku solutions of order $n^2$ that are both necessary and sufficient for determining an $\OOA (4,s,2,n)$ (via the identification in the proof of Theorem \ref{t:equivtosudoku}). A set of mutually orthogonal sudoku solutions satisfying these additional conditions will be called {\bf strongly orthogonal}.

We first introduce more terminology:  If $M$ is a sudoku solution, the corresponding {\bf radix solution}, denoted $R(M)$, consists of the radix digits of the symbols of $M$. The {\bf radix large rows (large columns)}, {\bf radix subsquares}, and {\bf radix rows (columns)} of $M$ are the large rows (large columns), subsquares, and rows (columns) of $R(M)$, respectively (Figure \ref{f:radix}).
\begin{figure}[h]
$$
M=\begin{array}{|cc|cc|}
 \hline
 0 & 1 & 3 & 2 \\
 2 & 3 & 1 &  0 \\
 \hline
 1 & 0 & 2 & 3\\
 3 & 2 & 0 & 1\\
 \hline
    \end{array}
    =
 \begin{array}{|cc|cc|}
 \hline
  00 & 01 & 11 & 10 \\
  10 & 11 & 01 & 00 \\
  \hline
 01 & 00 & 10 & 11\\
 11 & 10 & 00 & 01\\
 \hline
    \end{array}
 \quad \Longrightarrow \quad
 R(M)= \begin{array}{|cc|cc|}
 \hline
 0 & 0 & 1 & 1 \\
 1 & 1 & 0 &  0 \\
 \hline
 0 & 0 & 1 & 1\\
 1 & 1 & 0 & 0\\
 \hline
    \end{array}
       $$
\caption{A sudoku solution and its corresponding radix solution.}
\label{f:radix}
\end{figure}
Finally, two radix solutions $R(M_i)$ and $R(M_j)$ may be superimposed to produce a third solution $N_{ij}$, called a {\bf composite solution}, whose symbols are ordered pairs. (These symbols may be considered numerals in base $n$.) $N_{ij}$ may or may not or may not be a sudoku solution (Figure \ref{f:Nij}).
\begin{figure}[h]
$$
 R(M_1)= \begin{array}{|cc|cc|}
 \hline
 0 & 0 & 1 & 1 \\
 1 & 1 & 0 &  0 \\
 \hline
 0 & 0 & 1 & 1\\
 1 & 1 & 0 & 0\\
 \hline
    \end{array} \quad
 R(M_2)= \begin{array}{|cc|cc|}
 \hline
 0 & 1 & 1 & 0 \\
 1 & 0 & 1 &  0 \\
 \hline
 1 & 0 & 0 & 1\\
 0 & 1 & 0 & 1\\
 \hline
    \end{array} \Longrightarrow
 N_{12}= \begin{array}{|cc|cc|}
 \hline
 00 & 01 & 11 & 10 \\
 11 & 10 & 01 &  00 \\
 \hline
 01 & 00 & 10 & 11\\
 10 & 11 & 00 & 01\\
 \hline
    \end{array}
    =
 \begin{array}{|cc|cc|}
 \hline
 0 & 1 & 3 & 2 \\
 3 & 2 & 1 &  0 \\
 \hline
 1 & 0 & 2 & 3\\
 2 & 3 & 0 & 1\\
 \hline
    \end{array}
       $$
\caption{Two radix solutions and the corresponding composite solution, which, in this case, is a sudoku solution.}
\label{f:Nij}
\end{figure}

Applying the definition of top-justified (see Introduction) and extracting the sudoku top-justified rows described in the proof of Theorem \ref{t:equivtosudoku}, we obtain:
\begin{lem}\label{l:remaining}
In an $\SA (s,n)$, any set of top-justified rows that is {\em not} sudoku top-justified has exactly one of the following forms listed in {\rm (1a)} through {\rm (4a)} below, where throughout
$i,j,k,l$ are distinct members of $\{1,\dots, s-2\}$  and $R_{i,1}, R_{i,2}$ are as in Theorem \ref{t:equivtosudoku}. The array below is merely a visual representation of the sets listed in {\rm (1a)} through {\rm (4a)}, and no order is implied for indices $i,j,k,l$.

\begin{minipage}{2.5in}
\begin{itemize}
\item[(1a)] $\{l_{r,1},l_{c,1},l_{c,2},R_{i,1}\}$
\item[(1b)] $\{l_{r,1}, l_{r,2},l_{c,1},R_{i,1}\}$
\medskip

\item[(2a)] $\{l_{r,1}, l_{c,1},R_{i,1},R_{j,1}\}$
\item[(2b)] $\{l_{r,1}, l_{r,2},R_{i,1},R_{j,1}\}$
\item[(2c)] $\{l_{c,1}, l_{c,2},R_{i,1},R_{j,1}\}$
\item[(2d)] $\{l_{c,1}, R_{i,1},R_{i,2},R_{j,1}\}$
\item[(2e)] $\{l_{r,1}, R_{i,1},R_{i,2},R_{j,1}\}$
\medskip

\item[(3a)] $\{l_{r,1}, R_{i,1},R_{j,1},R_{k,1}\}$
\item[(3b)] $\{l_{c,1}, R_{i,1},R_{j,1},R_{k,1}\}$
\item[(3c)] $\{R_{i,1}, R_{j,1},R_{k,1},R_{k,2}\}$
\medskip

\item[(4a)] $\{R_{i,1}, R_{j,1},R_{k,1},R_{l,1}\}$
\end{itemize}
\end{minipage}
\begin{minipage}{2.5in}
$\begin{array}{|c|cc|ccccc|ccc|c|}
\hline
\ &1a & 1b & 2a & 2b & 2c & 2d & 2e & 3a & 3b & 3c & 4a \\
\hline
{\rm row }& * & * & * & * & \ & \ & * & * & \ & \ & \ \\
{\rm locs}& \ & * & \ & * & \ & \ & \ & \ & \ & \ & \ \\
 \hline
{\rm col}& * & * & * & \ & * & * & \ & \ & * & \ & \ \\
{\rm locs}& * & \ & \ & \ & * & \ & \ & \ & \ & \ & \ \\
 \hline
i{\rm -th}& * & * & * & * & * & * & * & * & * & * & *  \\
{\rm band}& \ & \ & \ & \ & \ & \ & \ & \ & \ & \ & \  \\
 \hline
j{\rm -th}& \ & \ & * & * & * & * & * & * & * & * & *  \\
{\rm band}& \ & \ & \ & \ & \ & * & * & \ & \ & \ & \  \\
 \hline
k{\rm -th}& \ & \ & \ & \ & \ & \ & \ & * & * & * & *  \\
{\rm band}& \ & \ & \ & \ & \ & \ & \ & \ & \ & * & \  \\
 \hline
l{\rm -th}& \ & \ & \ & \ & \ & \ & \ & \ & \ & \ & *  \\
{\rm band}& \ & \ & \ & \ & \ & \ & \ & \ & \ & \ & \  \\
 \hline
 \end{array}
   $
\end{minipage}

\noindent Further, non-repetition of $4$-tuples implied by each of the applicable top-justified sets listed above is necessary and sufficient to declare an $\SA (s,n)$ to be an $\OOA (4,s,2,n)$. (Here a top-justified set of rows is applicable if $s$ is large enough to make its existence possible. For example, $s\geq 5$ is required for sets {\rm (3a)} through {\rm (4a)} to exist.)
\end{lem}

Now we translate the non-repetition conditions described in Lemma \ref{l:remaining} to conditions on the corresponding collection of mutually orthogonal sudoku solutions obtained from an $\SA (s,n)$ via Theorem \ref{t:equivtosudoku}.

\begin{prop}\label{p:conditions}
Let $s,n\in \Z ^+$ with $s\geq 3$, and let $i,j,k,l\in \{1,\dots ,s-2\}$ be distinct. A collection $\{M_1,\dots ,M_{s-2}\}$ of mutually orthogonal sudoku solutions of order $n^2$ is strongly orthogonal if and only if the following conditions are satisfied for each choice of $i,j,k,l$:
\begin{itemize}
\item[(i)] Whenever $s\geq 3$ the subsquares of $R(M_i)$ are latin squares.
\item[(ii)] Whenever $s\geq 4$,
     \begin{itemize}
     \item[(a)] the composite solution $N_{ij}$ is a sudoku solution.
     \item[(b)] corresponding large rows of $R(M_i)$ and $M_j$ are orthogonal.
     \item[(c)] corresponding large columns of $R(M_i)$ and $M_j$ are orthogonal.
     \end{itemize}
\item[(iii)] Whenever $s\geq 5$,
     \begin{itemize}
     \item[(a)] corresponding large rows of  $N_{ij}$ and $R(M_k)$ are orthogonal.
     \item[(b)]  corresponding large columns of $N_{ij}$ and $R(M_k)$ are orthogonal.
     \item[(c)] $N_{ij}$ and $M_k$ are orthogonal.
     \end{itemize}
\item[(iv)] Whenever $s\geq 6$ the composite solutions $N_{ij}$ and $N_{kl}$ are orthogonal.
\end{itemize}
\end{prop}

\begin{proof}
Throughout we refer to the sets listed in Lemma \ref{l:remaining}. Observe that non-repetition of 4-tuples for sets of type (1a) and (1b) means that radix subsquares of $M_{i}$ must be latin squares, thus yielding item (i) above. Non-repetition of $4$-tuples for sets of type (2a), (2b), and (2c) says that corresponding  subsquares, rows, and columns of $R(M_{i})$ and $R(M_{j})$ must be orthogonal, respectively. In turn this means that the symbols for $N_{ij}$ are not repeated in any subsquare, row, or column, so we obtain item (ii)(a) above. Non-repetition of $4$-tuples in sets (2e) and (2f) implies that radix large rows of $M_{i}$ are orthogonal to ordinary large rows of $M_{j}$, and similarly for large columns, thus giving items (ii)(b) and (ii)(c) above. Non-repetition of $4$-tuples for set (3a) is equivalent to requiring corresponding large rows of $R(M_i)$, $R(M_j)$, and $R(M_k)$ to be orthogonal (here meaning no repetition of 3-tuples of symbols, {\em not} pairwise orthogonality). This is the same as saying that corresponding large rows of $N_{ij}$ and $R(M_k)$ are orthogonal, yielding (iii)(a) above. Similarly, non-repetition in set (3b) yields (iii)(b) above. Non-repetition of $4$-tuples in set (3c) says that $R(M_i)$, $R(M_j)$, and $M_k$ are orthogonal (again, we mean no repetition of symbol triplets), which is equivalent to requiring $N_{ij}$ and $M_k$ to be orthogonal, and so we obtain item (iii)(c) above. Finally, non-repetition in set (3d) says that $R(M_i)$, $R(M_j)$, $R(M_k)$, and $R(M_l)$ are orthogonal (meaning no repetition of symbol 4-tuples), and this is equivalent to requiring orthogonality of $N_{ij}$ and $N_{kl}$, giving item (iv) above.
\end{proof}

According to the proof of Proposition \ref{p:conditions}, any collection of corresponding radix subsquares within a set of strongly orthogonal sudoku solutions must be pairwise mutually orthogonal latin squares. Since there are at most $n-1$ mutually orthogonal latin squares of order $n$, this gives a crude upper bound on the size of a strongly orthogonal family of order-$n^2$ sudoku solutions:

\begin{cor}\label{c:crude}
The size of a strongly orthogonal family of sudoku solutions of order $n^2$ is no larger than $n-1$. Equivalently, if $\OOA (4,s,2,n)$ exists we must have $s\leq n+1$.
\end{cor}

One consequence of Corollary \ref{c:crude} is that arrays of type $\OOA (4,4,2,2)$ do not exist, so no sudoku array of type $\SA (4,2)$, such as that shown in Figure \ref{f:sudaltarray}, can ever be an ordered orthogonal array.

\section{Linear sudoku and flags}\label{s:linear}
Our overall goal is to construct ordered orthogonal arrays of type $\OOA (4,s,2,n)$ by constructing a families of $s-2$ strongly orthogonal sudoku solutions of order $n^2$. This means we want to produce a family of $s-2$ pairwise mutually orthogonal sudoku solutions that satisfy the conditions listed in Proposition \ref{p:conditions}. In this section we introduce a framework, known as {\bf linear sudoku}, that we will use later to construct strongly orthogonal families of sudoku solutions. Proofs have been omitted for results in this section that have been carefully described elsewhere, such as in \cite{rB08} and \cite{jL10}.

\subsection{Linear sudoku}
Let $\F$ be the finite field of order
$q$. Hereafter all linear and matrix algebra will be performed over $\F$. The set of locations within a sudoku solution of order $q^2$ can be
identified with the vector space $\F ^4$ over $\F$. This identification constitutes a slight variation of that
given in Section \ref{s:sudokuarray}, where $\F$ now takes the role of $\Z _n$: Each
location has an address $(x_1,x_2,x_3,x_4)$ (denoted $x_1x_2x_3x_4$ hereafter), where $x_1$ and $x_3$
denote the large row and column of the location, respectively,
while $x_2$ and $x_4$ denote the row within a large row and column within a large column of the
location, respectively. Rows can be labeled in
increasing lexicographic order from top to bottom starting from zero, while columns are labeled in increasing lexicographic order from left to right (see
Figure \ref{f:locationex}).

\begin{figure}[h]
      $$ {\small \begin{array}{| c  c  c | c c c | c c c|}
   \hline
      0 & 1 & 2 &  4 & 5 &3 & 8 & 6 & 7 \\
      3 & 4 & 5 &  7 & 8 &6 & 2 & 0 & 1^* \\
      6 & 7 & 8 &  1 & 2 &0 & 5 & 3 & 4 \\ \hline
      1 & 2 & 0 &  5 & 3 &4 & 6 & 7 & 8 \\
      4 & 5 & 3 &  8 & 6 &7 & 0 & 1 & 2 \\
      7 & 8 & 6 &  2 & 0 &1 & 3 & 4 & 5 \\ \hline
      2 & 0 & 1 &  3 & 4 &5 & 7 & 8 & 6 \\
      5 & 3 & 4 &  6 & 7 &8 & 1 & 2 & 0 \\
      8 & 6 & 7 &  0 & 1 &2 & 4 & 5 & 3 \\ \hline
    \end{array}}
      $$ 
       \caption{A linear sudoku solution generated by $\pt{1002,0212}$ with
       asterisked symbol in location $0122$.}  
       \label{f:locationex}
       \end{figure}

We say that a sudoku solution is {\bf linear} if the collection of locations housing any given symbol
is a coset of some two-dimensional vector subspace of $\F^4$. Linear sudoku solutions come in two flavors:
 If every such coset originates from a {\em single} two-dimensional
subspace, then the solution is of {\bf parallel type}; otherwise the
solution is of {\bf non-parallel type}. In this article we focus only
on linear sudoku solutions of parallel type, which we hereafter refer to simply as linear sudoku solutions. For example, the sudoku solution implied by Figure
\ref{f:locationex} is linear (of parallel type), generated by
$g=\pt{1002,0212}\subset \Z _3^4$.

In order to generate a
linear sudoku solution we require that cosets of $g$ intersect
each row, column, and subsquare exactly once:

\begin{prop} \label{p:sudokucharacterization}
{\rm \cite{rB08},\cite{jL10}} A two-dimensional subspace $g$ of $\F ^4$ generates a linear sudoku solution (up to labels) if and only if $g$ has trivial intersection with
$\pt{1000,0100}$, $\pt{0010,0001}$, and  $\pt{0100,0001}$.
\end{prop}

All linear sudoku solutions can be represented by
$2\times 2$ matrices. If  $A,B\in \MF$ we let $\left[
\begin{array}{c} A
\\ B
\end{array}\right]$ denote the subspace of $\F ^4$ spanned by the
columns of the matrix $\begin{pmatrix} A \\ B\end{pmatrix}$. Also
let $I$ denote the $2\times 2$ identity matrix. In consideration
of Proposition \ref{p:sudokucharacterization}, we have the
following:

\begin{prop}\label{p:sudmatrep} {\rm \cite{jL10}}
A two-dimensional subspace $g$ of $\F ^4$ generates a linear sudoku solution (up to labels) if and only if there exists a non-lower triangular
invertible matrix $\Gamma $ such that $g=\left[ \begin{array}{c} I\\ \Gamma
\end{array} \right]$.
\end{prop}

There is a simple geometric condition that characterizes
orthogonality of parallel linear sudoku solutions:

\begin{prop} \label{p:orthogcond} {\rm \cite{rB08},\cite{jL10}}
Let $M_{g},M_{h}$ be linear sudoku solutions generated by two-dimensional subspaces $g,h$ of $\F ^4$, respectively.
The two solutions are orthogonal if and only if $g$ and $h$ have trivial intersection.
\end{prop}

Proposition \ref{p:orthogcond} together with Proposition
\ref{p:sudmatrep} imply

\begin{cor}\label{c:orthogcond}
Two-dimensional subspaces $g_1=\left[ \begin{array}{c} I\\ \Gamma _1
\end{array} \right]$ and $g_2=\left[ \begin{array}{c} I\\ \Gamma _2
\end{array} \right]$ of $\F ^4$ generate orthogonal linear sudoku solutions
if and only if $\Gamma _1,\Gamma _2$ satisfy the conditions of
Proposition \ref{p:sudmatrep} and  $\det (\Gamma _1-\Gamma _2)\ne 0$.
\end{cor}

\subsection{Sudoku flags}
To produce strongly orthogonal families of sudoku solutions, it is desirable to have more control over the location of radix symbols than is provided by the linear sudoku framework presented above. (The need for control over the locations of radix symbols is evidenced by Proposition \ref{p:conditions}.) To place more control on the location of radix symbols, we introduce the notion of a {\bf sudoku flag}.

 Here a {\bf flag} $Z$ is a pair $(g,V)$ of subspaces of $\F^4$ where $g$ is two-dimensional, $V$ is three-dimensional, and $g\subset V$. If $Z=(g,V)$ and $g$ generates a linear sudoku solution, then we say $Z$ is a {\bf sudoku flag}. Each sudoku flag $Z$ generates a linear sudoku solution $M_{Z}$, where an assignment of symbols $\{0,1,\dots q^2-1\}$ to locations must satisfy:
 \begin{itemize}
 \item each coset of $g$ houses exactly one symbol, and
 \item each coset of $V$ houses exactly one radix symbol.
 \end{itemize}
These labeling restrictions may be rephrased as follows: symbols in $M_Z$ lie in cosets of $g$, while symbols in $R(M_Z)$ lie in cosets of $V$. If $v_1,v_2,v_3\in \F^4$ with $g=\langle v_1,v_2\rangle$ and $V=\langle v_1,v_2,v_3\rangle$, the corresponding flag $Z=(g,V)$ will be denoted $[v_1\ v_2\ v_3]$.

To illustrate a sudoku flag $Z$ and corresponding sudoku solution $M_Z$, let $\F =\Z _3$ and put
$$
Z=(g,V)=\left[ \begin{array}{ccc} 1 & 0 & 0 \\ 0 & 1 & 1 \\ 1 & 1 & 0 \\ 0 & 2 & 2 \end{array} \right].
    $$
Assigning the radix symbols $0,1,2$ to the three cosets of $V=\langle 1010, 0112,0102\rangle$ in $\Z _3^4$, we obtain, for example,
$$
R(M_Z)=\begin{array}{| c  c  c | c c c | c c c|}
   \hline
      0 & 1 & 2 &  0 & 1 &2 & 0 & 1 & 2   \\
      1 & 2 & 0 &  1 & 2 &0 & 1 & 2 & 0   \\
      2 & 0 & 1 &  2 & 0 &1 & 2 & 0 & 1 \\ \hline
      0 & 1 & 2 &  0 & 1 &2 & 0 & 1 & 2   \\
      1 & 2 & 0 &  1 & 2 &0 & 1 & 2 & 0   \\
      2 & 0 & 1 &  2 & 0 &1 & 2 & 0 & 1 \\ \hline
      0 & 1 & 2 &  0 & 1 &2 & 0 & 1 & 2   \\
      1 & 2 & 0 &  1 & 2 &0 & 1 & 2 & 0   \\
      2 & 0 & 1 &  2 & 0 &1 & 2 & 0 & 1 \\ \hline
\end{array}
$$
Then we assign symbols to the cosets of $g=\langle 1010, 0112\rangle$ in $\Z _3^4$, while preserving the radix symbol assignment above. For example, we obtain
$$
M_Z=\begin{array}{| c  c  c | c c c | c c c|}
\hline
      00 & 10 & 20 &  02 & 12 &22 & 01 & 11 & 21   \\
      11 & 21 & 01 &  10 & 20 &00 & 12 & 22 & 02   \\
      22 & 02 & 12 &  21 & 01 &11 & 20 & 00 & 10 \\ \hline
      01 & 11 & 21 &  00 & 10 &20 & 02 & 12 & 22   \\
      12 & 22 & 02 &  11 & 21 &01 & 10 & 20 & 00   \\
      20 & 00 & 10 &  22 & 02 &12 & 21 & 01 & 11 \\ \hline
      02 & 12 & 22 &  01 & 11 &21 & 00 & 10 & 20   \\
      10 & 20 & 00 &  12 & 22 &02 & 11 & 21 & 01   \\
      21 & 01 & 11 &  20 & 00 &10 & 22 & 02 & 12 \\ \hline
\end{array}\\
=
\begin{array}{| c  c  c | c c c | c c c|}
\hline
      0 & 3 & 6 &  2 & 5 &8 & 1 & 4 & 7   \\
      4 & 7 & 1 &  3 & 6 &0 & 5 & 8 & 2   \\
      8 & 2 & 5 &  7 & 1 &4 & 6 & 0 & 3 \\ \hline
      1 & 4 & 7 &  0 & 3 &6 & 2 & 5 & 8   \\
      5 & 8 & 2 &  4 & 7 &1 & 3 & 6 & 0   \\
      6 & 0 & 3 &  8 & 2 &5 & 7 & 1 & 4 \\ \hline
      2 & 5 & 8 &  1 & 4 &7 & 0 & 3 & 6   \\
      3 & 6 & 0 &  5 & 8 &2 & 4 & 7 & 1   \\
      7 & 1 & 4 &  6 & 0 &3 & 8 & 2 & 5 \\ \hline
\end{array}.
    $$
In general, such an assignment of symbols for $M_Z$ is always possible because each of the $q$ cosets of $V$ is partitioned by $q$ cosets of $g$.

We present two results about sudoku flags that will be used in later sections.

\begin{prop}\label{p:flagform}
A flag $Z=(g,V)$ is a sudoku flag and subsquares of $R(M_Z)$ are latin squares if and only if $Z$ can be written
$$
Z=\left[ \begin{array}{ccc} 1 & 0 & 0 \\ 0 & 1 & 1 \\ a & b & 0 \\ c & d & \beta \end{array} \right],
    $$
where $b$ and $\beta$ are nonzero and $\Gamma =\begin{pmatrix} a & b\\c &d \end{pmatrix}$ is nonsingular.
\end{prop}

\begin{proof}
Applying Proposition \ref{p:sudmatrep} and column operations,  $Z=(g,V)$ is a sudoku flag if and only if we can write
\begin{equation}\label{e:z}
Z=\left[ \begin{array}{ccc} 1 & 0 & 0 \\ 0 & 1 & \alpha \\ a & b & 0 \\ c & d & \beta \end{array} \right],
    \end{equation}
where $\Gamma$ is nonsingular and $b\ne 0$. To finish it suffices to show that subsquares of $R(M_Z)$ are not latin if and only if  at least one of $\alpha$ and $\beta$ are nonzero.
Suppose there are two distinct locations $m_1$ and $m_2$ in $R(M_Z)$ housing the same symbol and lying in the same row and subsquare. This is equivalent to the existence of $w\in \F^4$ and $v_1,v_2\in V$ with $m_1=w+v_1$, $m_2=w+v_2$, and $m_1-m_2 = (0,0,0,\mu)\in V$, where $\mu \ne 0$. The only way to obtain a vector of this form from the vectors in (\ref{e:z}) is as a multiple of $(0,\alpha ,0,\beta)$. We conclude that $\alpha =0$. Similarly, we the existence of two distinct locations $m_1$ and $m_2$ in $R(M_Z)$ housing the same symbol and lying in the same column and subsquare is equivalent to $\beta =0$.
\end{proof}

Proposition \ref{p:flagform} indicates that a sudoku flag giving rise to a sudoku solution with latin radix subsquares is completely determined by the datum $(\Gamma ,\beta)$, where
$\Gamma, \beta$ are as in the proposition.

\begin{lem}\label{l:coset}
Suppose $V_1,V_2$ are three-dimensional subspaces of $\F ^4$ with $V_1+V_2=\F ^4$. Then $V_1\cap V_2$ is two-dimensional and the intersection of a coset of $V_1$ with a coset of $V_2$ is a coset of $V_1\cap V_2$.
\end{lem}

\begin{proof}
Let $u_1+V_1$ and $u_2+V_2$ be cosets of $V_1$ and $V_2$. Since $V_1+V_2=\F ^4$ we may assume $u_1\in V_2$ and $u_2\in V_1$. Put $m=u_1+v_1=u_2+v_2\in (u_1+V_1)\cap (u_2+V_2)$, where $v_j\in V_j$ for $j=1,2$.

We show that
$(u_1+V_1)\cap (u_2+V_2)= (v_1+v_2)+(V_1\cap V_2)$. Let $n=u_1+w_1=u_2+w_2\in (u_1+V_1)\cap (u_2+V_2)$, where $w_j\in V_j$ for $j=1,2$. Note $n-m$ and $u_2-v_1$ lie in $V_1\cap V_2$, which implies
$$
n= m+(n-m)= (v_1+v_2)+[(u_2-v_1)+(n-m)]\in (v_1+v_2)+(V_1\cap V_2).
   $$
Now suppose $n=(v_1+v_2)+w\in (v_1+v_2)+(V_1\cap V_2)$ with $w\in V_1\cap V_2$. Note $v_2-u_1+w$ and $v_1-u_2+w$ lie in $V_1\cap V_2$, therefore
$$
n= u_1+ (v_1+(v_2-u_1+w))\in u_1+V_1 \text{ and } n= u_2+ (v_1+(v_2-u_2+w))\in u_2+V_2
   $$

Finally, the dimension formula says $V_1\cap V_2$ is two dimensional.
\end{proof}

\begin{prop}\label{p:linearcomposite}
Suppose $M_1$ and $M_2$ are orthogonal linear sudoku solutions generated by sudoku flags $Z_1=(g_1,V_1)$ and $Z_2=(g_2,V_2)$, respectively. If the composite solution
$N_{12}$ (see Section \ref{s:so}) is a sudoku solution, then it is a linear sudoku solution generated (up to labels) by the two dimensional subspace $g_{12}=V_1\cap V_2$.
\end{prop}

\begin{proof}
Let $b$ be a symbol in $N_{12}$ and let $b_1b_2$ be the base-$q$ representation of $b$. Locations in $N_{12}$ housing symbols with radix digit $b_1$ are the locations in $R(M_1)$ housing the symbol $b_1$. This set of locations has the form $u_1+V_1$ for some $u_1\in \F^4$. Likewise, locations in $N_{12}$ housing symbols with units digit $b_2$ are the locations in $R(M_2)$ housing the symbol $b_2$. This set of locations has the form $u_2+V_2$ for some
$u_2\in \F ^4$. Therefore the set of locations in $N_{12}$ housing the symbol $b$ has the form $(u_1+V_1)\cap (u_2+V_2)$. Further,  $V_1+V_2=\F ^4$ because $M_1$ and $M_2$ are orthogonal (due to Proposition \ref{p:orthogcond}). Applying  Lemma \ref{l:coset}, we see that the set of locations in $N_{12}$ housing the symbol $b$ is a coset of $g_{12}$.
\end{proof}

A computation of the intersection of the two three-dimensional subspaces of $\F ^4$ gives the following result.

\begin{prop} \label{p:gij}
Suppose sudoku flags $Z_1=(g_1,V_1)$ and $Z_2=(g_2,V_2)$, giving rise to orthogonal sudoku solutions with latin radix subsquares, possess data $(\Gamma _1, \beta _1)$ and $(\Gamma _2,\beta _2)$, respectively, with $\Gamma _i=\begin{pmatrix} a_i & b_i\\ c_i & d_i \end{pmatrix}$ for
$i\in \{1,2\}$. Further, assume $\beta _1 - \beta _2$ and $b_1(d_2-\beta _2)-b_2(d_1-\beta _1)$ are nonzero. Then
$$
g_{12}=V_1\cap V_2=\left[\begin{array}{cc} I \\ \Gamma _{12} \end{array}\right],
   $$
where $\Gamma _{12}=\begin{pmatrix} a_{12} & b_{12} \\ c_{12} & d_{12} \end{pmatrix}$, and
\begin{align*}
a _{12}&=\frac{b_1b_2(c_1-c_2)+a_2b_1(d_2-\beta _2)-a_1b_2(d_1-\beta _1)}{b_1(d_2-\beta _2)-b_2(d_1-\beta _1)},\\
b _{12}&=\frac{b_1b_2(\beta _1-\beta _2)}{b_1(d_2-\beta _2)-b_2(d_1-\beta _1)},\\
c _{12}&=\frac{b_1c_1(d_2-\beta _2)-b_2c_2(d_1-\beta _1)+(a_2-a_1)(d_1-\beta _1)(d_2-\beta _2)}{b_1(d_2-\beta _2)-b_2(d_1-\beta _1)}, \text{ and }\\
d _{12}&=\frac{\beta _1 b_1 (d_2-\beta _2)-\beta _2b_2(d_1-\beta _1)}{b_1(d_2-\beta _2)-b_2(d_1-\beta _1)}.\\
\end{align*}
\end{prop}

\section{Constructing ordered orthogonal arrays}\label{s:co}
In this section we recast Proposition \ref{p:conditions} in terms of sudoku flags. Then we use special collections of sudoku flags to construct families of strongly orthogonal sudoku solutions. This leads to several constructive existence theorems for ordered orthogonal arrays.

Throughout we use the notation of Section 4, and in addition we set $V_R=\langle 0100,0010,0001\rangle\subset \F ^4$ and $V_C=\langle 1000,0100,0001\rangle\subset \F ^4$. Note that $V_R$ and $V_C$ represent locations in the top large row and left large column, respectively, of sudoku solution of order $q^2$.

\begin{prop}\label{p:conditions2}
Let $s\in \Z ^+$ with $s\geq 3$, let $q$ be a prime power, and let $i,j,k,l\in \{1,\dots ,s-2\}$ be distinct. A collection $\{M_1,\dots ,M_{s-2}\}$ of mutually orthogonal sudoku solutions of order $q^2$, with $M_i$ arising from a sudoku flag $Z_i=(g_i,V_i)$ for each $i\in \{1,\dots ,s-2\}$, is strongly orthogonal if and only if the following conditions are satisfied for each choice of $i,j,k,l$:
\begin{itemize}
\item[(i)] Whenever $s\geq 3$, each $Z _i$ is determined by datum $(\Gamma _i ,\beta _i)$, with $\Gamma _i=\begin{pmatrix} a_i & b_i \\ c_i & d_i \end{pmatrix}$, as in Proposition \ref{p:flagform}.
\item[(ii)] Whenever $s\geq 4$,
     \begin{itemize}
     \item[(a)] $\Gamma _ {ij}$ is nonsingular and $b_{ij}$ is nonzero.
     \item[(b)] The matrix $\begin{pmatrix} 1 & 1 & 1 \\ b_i & 0 & b_j \\ d_i & \beta _i & d_j\end{pmatrix}$ is nonsingular.
     \item[(c)] The matrix $\begin{pmatrix} b_i\delta _i & 0 & b_j\delta _j \\ a_i\delta _i & \beta _i ^{-1} & a_j\delta _j \\ 1 & 1 & 1  \end{pmatrix}$ is nonsingular, where
        $\delta _k=[\det \Gamma _k]^{-1}$.
     \end{itemize}
\item[(iii)] Whenever $s\geq 5$,
     \begin{itemize}
     \item[(a)] $g_{ij}\cap (V_k\cap V_R)$ is trivial.
     \item[(b)] $g_{ij}\cap (V_k\cap V_C)$ is trivial.
     \item[(c)] $\Gamma _{ij}-\Gamma _{k}$ is nonsingular.
     \end{itemize}
\item[(iv)] Whenever $s\geq 6$, $\Gamma _{ij}-\Gamma_{kl}$ is nonsingular.
\end{itemize}
\end{prop}

\begin{proof}
We obtain item (i) from the corresponding statement in Proposition \ref{p:conditions} together with Proposition \ref{p:flagform}. Item (ii)(a) follows from the corresponding statement in Proposition \ref{p:conditions} together with Propositions \ref{p:gij} and \ref{p:sudmatrep}. Items (iii)(c) and (iv) are obtained from the corresponding statements in Proposition \ref{p:conditions} together with Proposition \ref{p:gij} and Corollary \ref{c:orthogcond}.

For item (iii)(a), observe that large rows of $N_{ij}$ and $R(M_k)$ are orthogonal if and only if for any $v,w,z\in \F ^4$, the cosets $v+g_{ij}$, $w+V_k$, and $z+V_R$ meet in exactly one location. By applying Lemma \ref{l:coset}, this is equivalent to saying that for any $v,\tilde w\in \F^4$ we have $|(v+g_{ij})\cap (\tilde w +(V_k\cap V_R))|=1$. In turn, by Lemma 3.1 of \cite{jL10}, this is equivalent to $g_{ij}\cap (V_k\cap V_R)=\{0\}$. Item (iii)(b) follows similarly.

Regarding item (ii)(b), the argument for item (iii)(a) above applies to show that large rows of $R(M_i)$ are $M_j$ are orthogonal exactly when $g_j\cap (V_i\cap V_R)=\{0\}$. Observe that
$$
V_i\cap V_R=\left[ \begin{array}{cc} 0 & 0 \\ 1 & 1 \\ b_i & 0 \\ d_i & \beta _i \end{array}\right] \text{ and }
g_j\cap V_R=\left[ \begin{array}{cc} 0  \\ 1  \\ b_j  \\ d_j \end{array}\right],
   $$
so that $g_j\cap (V_i\cap V_R)=\{0\}$ means $\begin{pmatrix} 1 & 1 & 1 \\ b_i & 0 & b_j \\ d_i & \beta _i & d_j\end{pmatrix}$ is nonsingular. Item (ii)(c) is verified similarly.
\end{proof}

Now we turn to applications. The main goal is to produce constructive existence theorems for ordered orthogonal arrays.

\subsection{Small values of $s$} Here we construct ordered orthogonal arrays of type $\OOA (4,s,2,q)$ when $s$ is small. Throughout $s,n$ is are positive integers with $s\geq 3$, and $q$ is a prime power.

We say that a collection of $s-2$ mutually orthogonal sudoku solutions of order $n^2$ is {\bf substrongly orthogonal} if it satisfies condition (i) of Proposition \ref{p:conditions} when $s\geq 3$ and also condition (ii) of Proposition \ref{p:conditions} when $s\geq 4$. Via the identification given in Theorem \ref{t:equivtosudoku}, a substrongly orthogonal set of sudoku solutions generates a sudoku array which, in general, has some, but not all of the non-repetition properties of a true ordered orthogonal array. When $s=3,4$ the notions of strong orthogonality and substrong orthogonality coincide.

\begin{thm}\label{t:substrong}
Let $q$ be a prime power. The maximum (achieved) size of a substrongly orthogonal set of sudoku solutions of order $q^2$ is $q-1$.
\end{thm}

\begin{proof}
Let $\F$ denote the field of order $q$. If $q=2$ then let $M$ be a sudoku solution of order $4$ generated by the sudoku flag  with data $(\Gamma ,\beta)$, where
$$
\Gamma = \begin{pmatrix} 1 & 1 \\ 0 & 1 \end{pmatrix} \text{ and } \beta = 1.
   $$
The set $\{M\}$ is substrongly orthogonal (in fact strongly orthogonal) because it meets condition (i) of Proposition \ref{p:conditions2}.

Now suppose $q>2$ and $\alpha\in \F ^\times$ with $\alpha \ne 1$. Let $\{M_i \mid i\in \F ^\times \}$ be a set of sudoku solutions of order $q^2$ generated by sudoku flags with
data $\{(\Gamma _i ,\beta _i)\mid i \in \F ^\times \}$ where
$$
\Gamma _i =  \begin{pmatrix} (\alpha  i)^{-1} & 1 \\ 0 & \alpha i \end{pmatrix} \text{ and } \beta _i = i.
    $$

Since $\Gamma_i -\Gamma _j$ is nonsingular for distinct members $i,j$ of $\F ^\times$, the sudoku solutions $\{M_i\}_{i\in \F^\times}$ are mutually orthogonal according to Corollary \ref{c:orthogcond}, and they satisfy item (i) of Proposition \ref{p:conditions2} by construction.

Using Proposition \ref{p:gij} we see that
$$
\Gamma _{ij}=\begin{pmatrix} 0 & (1-\alpha )^{-1} \\ \alpha ^{-1} (1-\alpha ) & 0 \end{pmatrix},
  $$
which is nonsingular and satisfies $b_{ij}\ne 0$, thereby satisfying item (ii)(a) of Proposition \ref{p:conditions2}. Further, items (ii)(b) and (ii)(c) of Proposition \ref{p:conditions2} are satisfied because the determinants of the matrices mentioned in those items, namely $\alpha (i-j)$ and $\alpha ^{-1} (j^{-1}-i^{-1})$, respectively, are both nonzero. We conclude that $\{M_i \mid i\in \F ^ \times\}$ is substrongly orthogonal (of size $q-1$).

Finally, by appealing to the proof of Proposition \ref{p:conditions}, corresponding order-$q$ radix subsquares of mutually orthogonal sudoku solutions satisfying items (i) and (ii) of Proposition \ref{p:conditions} must be mutually orthogonal latin squares. Such mutually orthogonal famulies of latin squares have size at most $q-1$.

\end{proof}

\begin{cor} \label{c:little}
For each prime power $q\geq 2$ there exists an $\OOA (4,3,2,q)$, and for $q\geq 3$ there exists an $\OOA (4,4,2,q)$.
\end{cor}

\begin{proof}
By Proposition \ref{p:conditions2}, any single element of a substrongly orthogonal set of sudoku solutions of order $q^2$ forms a strongly orthogonal set, as does any pair of elements. By Theorem \ref{t:substrong},  such single sudoku solutions exist whenever $q\geq 2$, and pairs exist when $q\geq 3$. Therefore, sudoku arrays obtained from these strongly orthogonal sets via the identification in Theorem \ref{t:equivtosudoku} form an $\OOA (4,3,2,q)$ and an $\OOA (4,4,2,q)$, respectively.
\end{proof}

To illustrate Theorem \ref{t:substrong} and Corollary \ref{c:little}, we put $\F =\Z _3$ and $\alpha =2$ in the proof of Theorem \ref{t:substrong}. We obtain data $(\Gamma _i , \beta _i)$ for $i\in\{1,2\}$ where $\beta _i=i$ and
$$
\Gamma _1 = \begin{pmatrix} 2 &  1 \\ 0 & 2 \end{pmatrix} \quad \text{and} \quad
\Gamma _2 = \begin{pmatrix} 1 &  1 \\ 0 & 1 \end{pmatrix}.
   $$
These data correspond to sudoku flags
$$
Z_1 =\left[ \begin{array}{ccc} 1 & 0 & 0 \\ 0 & 1 & 1 \\ 2 & 1 & 0 \\ 0 & 2 & 1 \end{array}\right] \quad \text{and}  \quad
Z_2=\left[ \begin{array}{ccc} 1 & 0 & 0 \\ 0 & 1 & 1 \\ 1 & 1 & 0 \\ 0 & 1 & 2 \end{array}\right],
    $$
which in turn give rise to the following strongly orthogonal pair of sudoku solutions, with symbols expressed in base $3$:

{\tiny
$$
M_1=\begin{array}{| c  c  c | c c c | c c c|}
\hline
      00 & 10 & 20 &  22 & 02 &12 & 11 & 21 & 01   \\
      21 & 01 & 11 &  10 & 20 &00 & 02 & 12 & 22   \\
      12 & 22 & 02 &  01 & 11 &21 & 20 & 00 & 10 \\ \hline
      22 & 02 & 12 &  11 & 21 &01 & 00 & 10 & 20   \\
      10 & 20 & 00 &  02 & 12 &22 & 21 & 01 & 11   \\
      01 & 11 & 21 &  20 & 00 &10 & 12 & 22 & 02 \\ \hline
      11 & 21 & 01 &  00 & 10 &20 & 22 & 02 & 12   \\
      02 & 12 & 22 &  21 & 01 &11 & 10 & 20 & 00   \\
      20 & 00 & 10 &  12 & 22 &02 & 01 & 11 & 21 \\ \hline
\end{array}
\quad \text{ and } \quad
M_2=
\begin{array}{| c  c  c | c c c | c c c|}
\hline
      00 & 10 & 20 &  12 & 22 &02 & 21 & 01 & 11   \\
      11 & 21 & 01 &  20 & 00 &10 & 02 & 12 & 22   \\
      22 & 02 & 12 &  01 & 11 &21 & 10 & 20 & 00 \\ \hline
      21 & 01 & 11 &  00 & 10 &20 & 12 & 22 & 02    \\
      02 & 12 & 22 &  11 & 21 &01 & 20 & 00 & 10   \\
      10 & 20 & 00 &  22 & 02 &12 & 01 & 11 & 21   \\ \hline
      12 & 22 & 02 &  21 & 01 &11 & 00 & 10 & 20    \\
      20 & 00 & 10 &  02 & 12 &22 & 11 & 21 & 01   \\
      01 & 11 & 21 &  10 & 20 &00 & 22 & 02 & 12     \\ \hline
\end{array}.
    $$
   }
Using the identification described in Theorem \ref{t:equivtosudoku}, solutions $M_1$ and $M_2$ give rise to the following $\OOA (4,4,2,3)$ (only partially shown):
{\tiny
$$
    \begin{array}{|ccccccccccccccccccc|}
 \hline
 0& 0 & 0 & 0 & 0 &  0 &  0 &  0 & 0 & 0 & 0 & 0 & 0 & 0 & 0 & 0 & 0 & 0 & \cdots \\
 0& 0 & 0 & 0 & 0 &  0 &  0 &  0 & 0 & 1 & 1 & 1 & 1 & 1 & 1 & 1 & 1 & 1 & \cdots \\
 \hline
 0& 0 & 0 & 1 & 1 &  1 &  2 &  2 & 2 & 0 & 0 & 0 & 1 & 1 & 1 & 2 & 2 & 2 & \cdots \\
 0& 1 & 2 & 0 & 1 &  2 &  0 &  1 & 2 & 0 & 1 & 2 & 0 & 1 & 2 & 0 & 1 & 2 & \cdots \\
 \hline
 0& 1 & 2 & 2 & 0 &  1 &  1 &  2 & 0 & 2 & 0 & 1 & 1 & 2 & 0 & 0 & 1 & 2 & \cdots  \\
 0& 0 & 0 & 2 & 2 &  2 &  1 &  1 & 1 & 1 & 1 & 1 & 0 & 0 & 0 & 2 & 2 & 2 & \cdots \\
 \hline
 0& 1 & 2 & 1 & 2 &  0 &  2 &  0 & 1 & 1 & 2 & 0 & 2 & 0 & 1 & 0 & 1 & 2 & \cdots  \\
 0& 0 & 0 & 2 & 2 &  2 &  1 &  1 & 1 & 1 & 1 & 1 & 0 & 0 & 0 & 2 & 2 & 2 & \cdots \\
 \hline
 \end{array}
   $$
}

\subsection{Larger values of $s$}

As before, let $q$ be a prime power and $s\in \Z$ with $s\geq 3$. For a given value of $q$ we present values of $s$ for which ordered orthogonal arrays of type $\OOA (4,s,2,q)$ must exist and indicate how they can be constructed.

\begin{thm}\label{t:big}
Let $\F$ be the field of order $q$, and let $S\subset \F ^\times$ be such that if $i,j\in S$ are distinct, then $i+j\ne 0$ and $ij\ne -1$. The collection of sudoku flags with data  $\{ (\Gamma _i ,\beta _i)\mid i\in S\}$, where
$$
\Gamma _i=\begin{pmatrix} i & 1 \\ 0 & i^{-1} \end{pmatrix} \quad \text{and}\quad \beta_i =i,
    $$
gives rise to a strongly orthogonal collection of order-$q^2$ sudoku solutions of size $|S|$.
\end{thm}

\begin{proof}
We need to verify that this data gives rise to mutually orthogonal sudoku solutions, and that the corresponding sudoku flags satisfy the conditions of Proposition \ref{p:conditions2}. Because $\Gamma _i$ is nonsingular for $i\in S$, $\Gamma _i -\Gamma _j$ is nonsingular for distinct $i,j\in S$, and $b_i$ is nonzero, Proposition \ref{p:sudmatrep} and Corollary \ref{c:orthogcond} indicate that the data give rise to a family of mutually orthogonal sudoku solutions. Further, by Proposition \ref{p:flagform}, condition (i) of Proposition \ref{p:conditions2} is satisfied, and a moment's work shows that conditions (ii)(b) and (ii)(c) are satisfied.

The remaining conditions in Proposition \ref{p:conditions2} require us to observe that for distinct $i,j\in S$,
$$
\Gamma _{ij}=\left(\begin{array}{cc} \frac{ij(i+j)}{1+ij} & \frac{ij}{1+ij} \\ \frac{(1-i^2)(j^2-1)}{1+ij} & \frac{i+j}{1+ij}   \end{array}\right).
    $$
Observe that $\Gamma _{ij}$ is well defined due to our hypotheses, that $b_{ij}$ is nonzero, and that, since $\det \Gamma _{ij}=ij\ne 0$, we know
$\Gamma _{ij}$ is nonsingular. We conclude that condition (ii)(a) of Proposition \ref{p:conditions2} is satisfied. Further, since
$\det (\Gamma _{ij}-\Gamma _k)$ and $\det (\Gamma _{ij}-\Gamma _{kl})$ are
$$
\frac{(i+j)(k-i)(j-k)}{k(1+ij)} \quad \text{and}\quad \frac{(i-k)(j-k)(i-l)(j-l)}{(1+ij)(1+kl)},
    $$
respectively, we apply the hypotheses to conclude that conditions (iii)(c) and (iv) of Proposition \ref{p:conditions2} are satisfied.

Next, observe that $g_{ij}\cap V_R =\langle (0, 1, \frac{ij}{1+ij},\frac{i+j}{1+ij})\rangle$ and $V_k\cap V_R=\langle (0,1,1,k^{-1}),(0,1,0,k)\rangle$. Therefore
$g_{ij}\cap (V_k\cap V_R)$ is trivial if and only if
$$
A_R= \begin{pmatrix} 1 & 1 & 1 \\ \frac{ij}{1+ij} & 1 & 0 \\ \frac{i+j}{1+ij} & k^{-1} & k \end{pmatrix}
    $$
is nonsingular. But $\det A_R=\frac{(i-k)(j-k)}{k(1+ij)}\ne 0$ by hypothesis, hence condition (iii)(a) of Proposition \ref{p:conditions2} holds.

Similarly, $g_{ij}\cap V_C=\langle (\frac{-1}{1+ij},\frac{i+j}{1+ij},0,1) \rangle$ and $V_k\cap V_C =\langle (1,-k,0,-1),(0,1,0,k) \rangle $. Therefore
$g_{ij}\cap (V_k\cap V_C)$ is trivial if and only if
$$
A_C=\begin{pmatrix} \frac{-1}{1+ij} &  1 &  0 \\ \frac{i+j}{1+ij} & -k & 1 \\ 1 & -1 & k    \end{pmatrix}
   $$
is nonsingular. But $\det A_C=\frac{(i-k)(j-k)}{1+ij}\ne 0$ by hypothesis, hence condition (iii)(b) of Proposition \ref{p:conditions2}  holds.
\end{proof}

\begin{cor}\label{c:big}
Let $q$ be a prime power. Ordered orthogonal arrays of type $\OOA (4,s,2,q)$ may be constructed whenever $\displaystyle 3\leq s\leq \frac{q+4}{2}$.
\end{cor}

\begin{proof}
A strongly orthogonal collection of order-$q^2$ sudoku solutions of size $|S|$ constructed via Theorem \ref{t:big} gives rise via the identification of Theorem \ref{t:equivtosudoku} to
an $\OOA (4,|S|+2,2,q)$. Therefore it suffices to show that the set $S$ defined in Theorem \ref{t:big} can be chosen so that $|S|=\frac{q}{2}$.

If $q$ is even, then the conditions for membership in $S$ are equivalent to requiring that if $i,j\in S$ are distinct then $ij\ne 1$. Let $1\in S$, and then for each of the remaining $q-2$ members $i$ of $\F ^\times $, select either $i$ or $i^{-1}$ for inclusion in $S$, but not both. Then $|S|=1+(q-2)/2=q/2$, as desired.

If $q$ is odd, let $\alpha\in \F^\times$ be a multiplicative generator for $\F ^\times$ and observe that $\Z _{q-1}\simeq \F^\times$ via $k\mapsto \alpha ^k$. It then suffices to produce $\tilde S\subset \Z_{q-1}$, with
$|\tilde S|=q/2$, where
$k\in \tilde S$ means that $k+\frac{q-1}{2}$ is excluded from $\tilde S$ when $k+\frac{q-1}{2}$ is distinct from $k$, and that $-(k+\frac{q-1}{2})$ is excluded from $\tilde S$ when $-(k+\frac{q-1}{2})$ is distinct from $k$.
In the case that $q-1=4n+2$, we may choose $\tilde S=\{0,\pm 1,\dots ,\pm n\}$, while in the case $q-1=4n$ we may choose $\tilde S=\{0,\pm 1, \dots ,\pm (n-1),n\}$. In both cases, $|\tilde S|=q/2$, as desired.
\end{proof}

Observe that the proof of Theorem \ref{t:big} assumes $|S|\geq 4$, and so is best suited for producing ordered orthogonal arrays of type $\OOA (4,s,2,q)$ where $s\geq 6$. For this reason, Corollary \ref{c:big} does not give best results when $q$ is small (e.g., $q=3,4,5$).

\section{Concluding remarks}

We present two concluding remarks suggesting further study:
First, Corollaries \ref{c:big} and \ref{c:crude} show that when $q$ is a prime power, the maximum value of $s$ for which $\OOA (4,s,2,q)$ exists lies somewhere between $ \frac{q+4}{2}$ and $q+1$. It is possible that the ideas in this paper can be used to further narrow this range of maximum values for $s$.
Second, our results rely heavily on the notion of sudoku flags, which suggests that, generally speaking, flag geometry may play a role in constructing ordered orthogonal arrays that is analagous to the role projective geometry (or, more properly, the geometry of Bruck nets) plays in constructing ordinary orthogonal arrays.

\end{document}